%% file: MasterFile_MutualStat.tex
\documentclass[11pt]{amsart}  

\usepackage{lineno}
\usepackage[top=1.4in, bottom=1.4in, left=1.5in, right=1.5in]{geometry}

\usepackage{amsrefs}
\usepackage[all]{xy}
\usepackage{syntonly}
\usepackage{enumerate}
\usepackage{hyperref}
\usepackage{amsfonts}
\usepackage{amssymb}
\usepackage{amsmath}
\usepackage{amsthm}
\usepackage{enumitem}

\usepackage{standalone}

\usepackage{tikz}

\usepackage{graphics}
\usepackage{graphicx}
\usepackage{color}
\usepackage{comment}

\newtheorem{theorem}{Theorem}
\newtheorem{corollary}[theorem]{Corollary}

\newtheorem{definition}[theorem]{Definition}
\newtheorem{lemma}[theorem]{Lemma}
\newtheorem{globalClaim}[theorem]{Claim}

\newtheorem{observation}[theorem]{Observation}
\newtheorem{question}[theorem]{Question}

\newtheorem{fact}[theorem]{Fact}

\newtheorem{remark}[theorem]{Remark}
\newtheorem{notation}[theorem]{Notation}

\specialcomment{com}{ \color{red} NOTE TO SELF: }{\color{black}} 
\specialcomment{dtos}{ \color{red} NOTE TO SEANIEBABY:} {\color{black}}
\specialcomment{com2dom}{ \color{red} NOTE TO DOMINIK: }{\color{black}}

\excludecomment{dtos} 
\excludecomment{com2dom}

\newcommand{\<}{\langle}
\renewcommand{\>}{\rangle}
\renewcommand{\phi}{\varphi}
\newcommand{\Gdw}{\Leftrightarrow}

\newcommand{\op}[1]{\operatorname{#1}}

\newcommand{\ZFC}{\op{ZFC}}

\newcommand{\SCH}{\op{SCH}}

\newcommand{\degree}{\op{deg}}

\newcommand{\ult}{\op{Ult}}
\newcommand{\on}{\op{On}}

\newcommand{\card}{\op{Card}}
\newcommand{\cof}{\op{cof}}
\newcommand{\pcf}{\op{pcf}}

\newcommand{\mitord}{\op{o}}

\newcommand{\Hull}{\op{Hull}}

\newcommand{\eextend}{\trianglelefteq}

\begin{document}

\title[]{Lower consistency bounds for mutual stationarity with divergent uncountable cofinalities}

\author{Dominik Adolf}
\email{d$\_$adol01@uni-muenster.de}

\author{Sean Cox}
\email{scox9@vcu.edu}
\address{
Department of Mathematics and Applied Mathematics \\
Virginia Commonwealth University \\
1015 Floyd Avenue \\
Richmond, Virginia 23284, USA 
}

\author{Philip Welch}
\email{P.Welch@bristol.ac.uk}

\thanks{This work was partially supported by a grant from the Simons Foundation (\#318467 to Sean Cox).}

\thanks{The first author was supported by Deutsche Forschungsgemeinschaft Research Fellowship AD 469/1-1.}

\subjclass[2010]{2010}

\begin{abstract}
We prove that the upper bounds for the consistency strength of certain instances of mutual stationarity considered by Liu-Shelah~\cite{MR1469093} are close to optimal. We also consider some related and, as it turns out, stronger properties. 
\end{abstract}

\maketitle

\input{sec_Intro.tex}

\input{sec_Prelims.tex}

\input{sec_ProofOfMainTheorem.tex}

\input{sec_FurtherTheorems.tex}
\input{sec_OpenProblems.tex}

\end{document}

%% file: sec_Intro.tex
\section{Introduction}

Mutual stationarity was originally introduced in \cite{MR1846032} to study saturation properties of non-stationary ideals.

\begin{definition}
 Let $\lambda$ be an ordinal, any ordinal. Let $\<\kappa_i : i < \lambda\>$ be an increasing sequence of regular cardinals, $\bar{\kappa} := \sup\limits_{i < \lambda} \kappa_i$. We say a sequence $\<S_i : i < \lambda\>$, where $S_i \subseteq \kappa_i$ is stationary, is mutually stationary iff the set 
 \[ \{ A \subset \bar{\kappa} \vert \forall i < \lambda: \kappa_i \in A \Rightarrow \sup(A \cap \kappa_i) \in S_i\} \]
 is stationary, i.e. contains a substructure of every structure with countable signature on $\bar{\kappa}$.
\end{definition}

The most remarkable result from the above paper is the $\ZFC$ fact that any sequence of stationary sets all of which concentrate on points of countable cofinalities is mutually stationary, no matter its length. It is also shown that an analog theorem for sets concentrating on cofinality $\omega_1$ can not be proven in $\ZFC$. We do not currently know if it is even consistent, though a lower bound for it's consistency is known (see \cite{MR2267155},\cite{MR2803948}).

However, we are only going to discuss sequences that do not concentrate on a fixed cofinality. We shall also limit ourselves to stationary subsets of the $\aleph_n$'s, $n$ a natural number.

All mutually stationary sequences appearing in the paper will have limit length.

We will mention some prior results involving sets concentrating on countable cofinality to draw some parallels with the results from this paper. We start with this result:

\begin{theorem}[Magidor]
 Let $\<\kappa_i : i < \omega\>$ be an increasing sequence of measurable cardinals. Then there exists a generic extension of the universe $V\left[G\right]$ in which $\kappa_i$ becomes $\aleph_{2{i+1}}$ and the sequence $\<S^2_0,S^3_1,S^4_0,S^5_1,S^6_0,\ldots\>$ is mutually stationary.
\end{theorem}

(Note: We will often use the following notation: $S^m_n := \{ \alpha < \aleph_m \vert \cof(\alpha) = \aleph_n \}$.)

The theorem can be stated more generally, the real limitations being that all but finitely many sets in the sequence concentrate on one of two cofinalities, $\omega$ being one of them, and if a set in the sequence concentrates on countable cofinality then the next one does not. Here the points in the sequence concentrating on countable cofinality correspond to former measurable cardinals and their successors are not collapsed in the construction and correspond to points concentrating on the other cofinality.

If one wants to do away with this non-accumulation property of points concentrating on countable cofinality one uses supercompact cardinals instead \footnote{Successors of supercompacts might be collapsed in this construction, this being a classic use of indestructability.}. In that case for any given $f:\omega \rightarrow 2$ in the ground model there exists a generic extension in which the sequence $\vec{S}^f_2 := \<S^m_{f(m)}: 2 \leq m < \omega\>$ is mutually stationary (see \cite{MR2250537}).

This result can be improved using a competely different approach. Jensen has shown the consistency of a forcing axiom (relative to one supercompact) that implies the mutual stationarity of $\vec{S}^f_2$ for all $f:\omega \rightarrow 2$ simultaneously. (See \cite{Jensen_FA_CH}.)

The Magidor result, too, can be improved:

\begin{theorem}[Koepke]
Let $\kappa$ be a measurable cardinal. Then there exists a generic extension of the universe $V\left[G\right]$ in which $\kappa$ becomes $\aleph_\omega$ and the alternating sequence $\<S^2_0,S^3_1,S^4_0,S^5_1,S^6_0,\ldots\>$ is mutually stationary.
\end{theorem}

It is not hard to see that this is optimal. There is an interesting switch that happened here. In the Koepke result different limitations apply: all but finitely many sets in the sequence concentrate on one of two cofinalities, $\omega$ being one of them, and if a set in the sequence concentrates on the \emph{other} cofinality then the next one does not. Here the points in the sequence concentrating on the \emph{other} cofinality correspond to points in a Prikry sequence and their successors are not collapsed in the construction and correspond to points concentrating on \emph{countable} cofinality. (See \cite{MR2276663}.)

This leads us to ask the following question. Is it possible to force the mutual stationarity of the sequence $\<S^2_0,S^3_1,S^4_1,S^5_0,S^6_1,S^7_1,\ldots\>$ from finitely many measurable cardinals?

From now on, all sets in a mutually stationary sequence will concentrate on points of uncountable cofinality. The following result is an analog to Magidor's result above.

\begin{theorem}[Liu-Shelah]\label{thm_LiuShelah_1}
 Let $1 \leq m < k$ be natural numbers. Let $A \subset \omega$ be infinite s.t. 
 \[ n \in A \Rightarrow n + 1 \notin A\]
 for all $n < \omega$. Let $f:\omega \rightarrow \{n,k\}$ be defined by 
 \[ f(n) :=\begin{cases} m & n \in A \\ k & n \notin A.\end{cases} \]  
 Let $\<\kappa_i : i < \omega\>$ be an increasing sequence of cardinals of Mitchell order at least $\omega_m + 1$. Then there exists a generic extension in which $\<\kappa_i : i < \omega\>$ is the increasing enumeration of $\<\aleph_n : n > k, n \in A\>$ and the sequence $\vec{S}^f_{k +1} := \< S^n_{f(n)} : k < n < \omega\>$ is mutually stationary.
\end{theorem}

(There was a significantly weaker precursor result in \cite{KLiuOldPaper}, but it has been superceded by this one from \cite{MR1469093}.)

We do not know about a higher level analog to the Koepke result, but we think that it should exist. 

The Liu-Shelah paper \cite{MR1469093} has another result, one which is nominally very powerful.

\begin{theorem}[Liu-Shelah]\label{PCF}
 Assume $\max(\pcf(\{\aleph_n: n <\omega\})) = \aleph_{\omega + n^*}$. Let $1 < m^* < \omega$. Let $I$ be the ideal of finite subsets of $\omega$. Let $\<A_i:i < n^*\>$ be a partition of $\omega$ such that $\prod\limits_{k \in A_i} \aleph_k \slash I$ has true cofinality $\aleph_{\omega + {i + 1}}$ for $i < n^*$. Let $\<n_i: i < n^*\> \subset \left[1,m^*\right]$ be arbitrary. Define a function $f:\omega \rightarrow \omega$ by
 \[ f(n) = n_i :\Gdw n \in A_i. \]
 Then the sequence $\vec{S}^f_{m^* + 1}$ is mutually stationary.
\end{theorem}

Note that the requirement here is the failure of $\SCH$ at $\aleph_\omega$. So, we are still below $0^\P$. We are interested to know if this theorem can be used to generate mutually stationary sequences not already covered by Theorem \ref{thm_LiuShelah_1}. For that end we do need the ability to control for the partition $\<A_i:i < n^*\>$\footnote{Doing so might necessitate large cardinals beyond $0^\P$.}. Unfortunately, we do not know how to do that. (See also Question \ref{question_5}.)

We now state the main results of this paper.  Theorem \ref{thm_MainTheorem} shows that the upper bounds obtained by Liu-Shelah in Theorem \ref{thm_LiuShelah_1} are close to optimal:

\begin{theorem}\label{thm_MainTheorem}
Let $1 < m$ be a natural number. Suppose $\langle S_n \ | \ n \ge m + 1 \rangle$ is a mutually stationary sequence such that:
\begin{enumerate}
 \item  for every $n \ge m + 1$, $S_n$ is stationary in $\omega_n$ and concentrates on a fixed cofinality $\mu_n$; 
 
 \item $\langle \mu_n \ |\ n \ge m + 1 \rangle$ is \textbf{not} eventually constant; and
 \item $\omega_1 \le \mu^*:= \liminf_{n \to \infty} \mu_n < \aleph_\omega$.
\end{enumerate}

Then there is an inner model $W$ such that: for infinitely many $n \in \omega$:
\[
V \models \big\{ \alpha < \omega_n \ | \ o^W(\alpha) \ge \mu^* \big\} \text{ is stationary in } \omega_n
\]

\end{theorem}

The hypotheses of Theorem \ref{thm_MainTheorem} are consistent, by the Liu-Shelah Theorem \ref{thm_LiuShelah_1}.  For example, mutual stationarity of the sequence
\begin{equation*}
\langle S^3_2, S^4_1, S^5_2, S^6_1, \dots, S^{2k-1}_2, S^{2k}_1, \dots \rangle
\end{equation*}
falls under the hypothesis of Theorem \ref{thm_MainTheorem} (with $\liminf_{n \to \infty} \mu_n = \omega_1$).

We do not know if the hypotheses of these following theorems is consistent. Theorem \ref{thm_MainTheorem_two} and Theorem \ref{thm_MainTheorem_three} have analogs in the countable case, mentioned in the introduction, which we do know to be consistent. Therefore we are confident that these hypotheses will be found to be consistent in the end. We are less confident about Theorem \ref{thm_MainTheorem_four}, but will include it anyway as it presents only a minimal time investment.

Furthermore, these hypotheses cover the most obvious variations of the hypothesis of our main theorem, Theorem \ref{thm_MainTheorem}. We feel the paper would be incomplete without addressing them.

\begin{theorem}\label{thm_MainTheorem_two}
Assume $0^\P$ does not exist. Fix natural numbers $l > 1, m > 1$. Suppose $\langle S_n \ | \ n \ge m + 1 \rangle$ is a sequence such that for every $n \ge m + 1$:
\begin{enumerate}
 \item $S_n$ is stationary in $\omega_n$ and concentrates on a fixed uncountable cofinality $\mu_n$; and there exists a strictly increasing sequence $\langle n_k : k < \omega \rangle$ with
 \item $n_{k+1} \geq n_k + l$ for all $k < \omega$
 \item $\langle \mu_{n_k} \ |\ k < \omega \rangle$ is not eventually constant 
 \item $\mu_{n_k} = \mu_{n_k + i}$ for all $k < \omega$ and $i < l$
 \item $\langle S_n \ | \ n \ge m + 1 \rangle$ is mutually stationary.
\end{enumerate}
Then in $K$ there is an infinite sequence $\<\kappa_n: n < \omega\> \subset \{\aleph_n : n < \omega\}$ s.t for all $n < \omega$ there is $\kappa < \kappa_n$ s.t $(\kappa^+)^K < \kappa_n$ and $\mitord^K(\kappa) \geq (\kappa_n)^{+(l-1)}$.
\end{theorem}

\begin{theorem}\label{thm_MainTheorem_three}
Let $1 \leq n,k < m < \omega$ and assume that the sequence $\langle S^{n+m}_{f(n)} : n < \omega\> \rangle$ is mutually stationary for all $f:\omega \rightarrow \{n,k\}$. Then $0^\P$ exists.
\end{theorem}

By the results of Liu-Shelah mentioned in Theorem \ref{thm_LiuShelah_1}, our Theorem \ref{thm_MainTheorem} is almost an equiconsistency.  However, if we alter the assumption of Theorem \ref{thm_MainTheorem} to require that $\liminf_{n \to \infty} \mu_n = \aleph_\omega$, the consistency strength jumps considerably, as shown by the following Theorem \ref{thm_MainTheorem_four}.  In fact, the hypotheses of Theorem \ref{thm_MainTheorem_four} is an apparent strengthening of stating that $\aleph_\omega$ is a Jonsson cardinal, which is not known to be consistent at all.

\begin{theorem}\label{thm_MainTheorem_four}
Fix $1 \leq m < \omega$. Suppose $\langle S_n \ | \ n \ge m \rangle$ is a mutually stationary sequence such that for every $n \ge m$:
\begin{enumerate}
 \item $S_n$ is stationary in $\omega_n$ and concentrates on a fixed cofinality $\mu_n$;
 \item $\liminf_{n \to \infty} \mu_n = \aleph_\omega$.\footnote{Equivalently, any given cofinality appears only boundedly often in $\<\mu_n: n \ge m\>$.}
\end{enumerate} 
Then $0^\P$ exists.
\end{theorem}

%% file: sec_Prelims.tex
\section{Preliminaries}\label{sec_Prelims}

\subsection{Inner model theory}

Unless otherwise stated, we follow the conventions of Zeman~\cite{ZemanBook}, assume that $0^\P$ does not exist, and let $K$ denote the core model (see Chapter 8 of \cite{ZemanBook}).  Like \cite{ZemanBook}, we use Jensen indexing of extenders.  We will heavily depend on the following lemma.

\begin{lemma}\label{lem_teclemma}
 Let $M$ be a premouse. Let $n$ and $\kappa$ be such that $M$ is $n + 1$-sound above $\kappa$. Assume that $\lambda \in M$ is such that 
 \[ \kappa < \lambda \leq \rho^M_n \]
 and $\cof^M(\lambda) > \kappa$. Then $\cof^V(\lambda) = \cof^V(\rho^M_n)$.
\end{lemma}

\begin{proof}
 We can assume that $n = 0$, otherwise replace $M$ by its $n$-th reduct. Define $f:\on \cap M \rightarrow \lambda$ by 
 \[ \xi \mapsto \sup(\Hull^{M || \xi}_1(\kappa \cup \{p^M_1\}) \cap \lambda). \]
 By assumption this is well-defined and cofinal. It is also clearly increasing. Hence, we are done.
\end{proof}

We will need the following basic fact about normal fine-structural iterations.
\begin{fact}[See Lemma 4.2.2 of \cite{ZemanBook}]\label{fact_DegreesStabilize}
 Suppose $\langle M_i \ | \ i \le \theta \rangle$ is a normal fine-structural iteration of a premouse $M = M_0$.  Let $\kappa_i$ denote the critical point of the $i$-th stage.  Assume that the ultimate projectum of $M_0$ is $\le \kappa_0$.   Then for every $i < \theta$, the ultimate projectum of $M_i$ is $\le \kappa_i$.   Let $\text{deg}(M_i,\kappa_i)$ denote the maximal $n \in \omega$ such that $\kappa_i < \omega \rho^M_n$.   If $\theta$ is a limit ordinal, then $\langle \text{deg}(M_i,\kappa_i) \ | \ i < \theta \rangle$ is eventually constant. 
\end{fact}

\subsection{Facts about mutual stationarity}\label{sec_Prelim_MutualStat}

The following lemma will be used to modify the members of sets witnessing mutual stationarity:
\begin{lemma}\label{lem_AddOrdsToBottom}
Suppose $\langle S_n \ | \ n \ge n_0 \rangle$ is a sequence such that  $S_n$ is a stationary subset of $\omega_n$ for every $n \ge n_0$.  Fix an algebra $\mathfrak{A} = (H_{\aleph_{\omega+1}}, \in, \dots)$ and assume that $X \prec \mathfrak{A}$ and $\text{sup}(X \cap \omega_n) \in S_n$ for every $n \ge n_0$.  Fix a regular uncountable $\mu < \aleph_\omega$ and set
\begin{equation*}
X':=  \text{Sk}^{\mathfrak{A}}\big( X \cup \mu \big)
\end{equation*}

Then for all $n$ such that $\omega_n > \mu$: 
\begin{equation*}
\text{sup}(X' \cap \omega_n) = \text{sup}(X \cap \omega_n) 
\end{equation*}
\end{lemma}
\begin{proof}
The $\ge$ direction is trivial.  For the $\le$ direction, fix an $n$ such that $\mu < \omega_n$.  Let $\eta$ be an element of $\omega_n \cap X'$.  Then there is a function $f \in X$ and an ordinal $\xi < \mu$ such that $\eta = f(\xi)$.  Let $h$ be the restriction of $f$ to those inputs from $\mu$ whose outputs are below $\omega_n$.  Since $\mu$ is among the $\aleph_k$'s then $\mu \in X$, and so since $f \in X$ it follows that $h \in X$.  Since $\omega_n$ is regular and $> \mu$ then $\text{sup}(\text{range}(h) ) \in X \cap \omega_n$.  Thus $\eta = f(\xi) = h(\xi) < \text{sup}( X \cap \omega_n)$.
\end{proof}

\begin{corollary}\label{cor_ReplaceBottom}
Suppose $\vec{S} = \langle S_n \ | \ n \ge n_0 \rangle$ is mutually stationary, where $S_n \subset \omega_n$ for each $n \ge n_0$.  Let $\mu < \aleph_\omega$ be fixed, and let $n_1$ be such that $\mu < \omega_{n_1}$.  Then the mutual stationarity of $\langle S_n \ | \ n \ge n_1 \rangle$ is witnessed by models which contain $\mu$ as a subset.  
\end{corollary}

The following lemma can be easily proved by induction on $n$:
\begin{lemma}\label{lem_Unif_Implies_Cover}
Assume $\mu < \aleph_\omega$ is regular, $\mu \subset X \prec H_{\aleph_{\omega+1}}$, and $\text{sup}(X \cap \omega_n)$ has cofinality $\ge \mu$ whenever $\omega_n \ge \mu$.  Then for every such $n$, every $<\mu$-sized subset of $X \cap \omega_n$ is covered by some $<\mu$-sized set from $X$. In particular, $X \cap \aleph_\omega$ is a $<\mu$-closed set of ordinals.
\end{lemma}

%% file: sec_ProofOfMainTheorem.tex
\section{Proof of Theorem \ref{thm_MainTheorem}}

In this section we prove Theorem \ref{thm_MainTheorem}.  Define
\begin{equation*}
\mu^*:= \liminf_{n \to \infty} \mu_n
\end{equation*}

Recall we are assuming that
\begin{equation}\label{eq_MuStarLessOmega}
\mu^* < \aleph_\omega
\end{equation}

\begin{remark}
The case where $\mu^* = \aleph_\omega$ is Theorem \ref{thm_MainTheorem_four}.  However, unlike the assumptions of Theorem \ref{thm_MainTheorem}, the assumptions of Theorem \ref{thm_MainTheorem_four} are not known to be consistent.
\end{remark}

As described in Section \ref{sec_Prelims}, we work with the core model $K$ below $0^\P$.\footnote{If $0^\P$ exists then by iterating $0^\P$ one easily obtains an inner model as in the conclusion of Theorem \ref{thm_MainTheorem}.}

First we state a couple of theorems which are proved in  \cite{CoxCoveringPaper}:
\begin{theorem}[\cite{CoxCoveringPaper}, Lemma 44]\label{thm_MitchellTheorem}
Let $K$ be the core model below 0-pistol and $\lambda$ an uncountable cardinal.  Assume $S$ is a stationary collection of $X \prec H_{\lambda}$ such that 
\begin{equation*}
\text{cof}(\omega) \cap \lambda \cap \text{lim}(X \cap \lambda)  \subset X
\end{equation*}
For each $X \in S$ let $\sigma_X: H_X \to X \prec H_\lambda$ be the inverse of the Mostowski collapse of $X$, and let $K_X:= \sigma_X^{-1}[K\cap H_\lambda]$.  Then for all but nonstationarily many $X \in S$, in the coiteration of $K$ with $K_X$:
\begin{itemize}
 \item The $K$ side truncates to a mouse of size at most $|\text{crit}(\sigma_X)|$ by stage 1 of the coiteration;
 \item the $K_X$ side of the coiteration is trivial.
\end{itemize} 
\end{theorem}

\begin{notation}\label{notation_IterationStuff}
Let $S$ be as in the hypothesis of Theorem \ref{thm_MitchellTheorem}.  For each $X \in S$ we let $\theta_X$ denote the length of the $K$ versus $K_X$ coiteration, and let $\langle N^X_i, \kappa^X_i, E^X_i \ | \ i < \theta_X \rangle$ denote the sequence of mice, critical points, and applied extenders on the $K$ side of the coiteration.\footnote{Recall from Theorem \ref{thm_MitchellTheorem} that the $K_X$ side of the coiteration is trivial.}  For $i \le j < \theta_X$ let $\pi^X_{i,j}$ denote the (possibly partial) iteration map from $N^X_i \to N^X_j$. 
\end{notation}

The following theorem was a generalization of a Covering Theorem of Mitchell:\footnote{E.g.\ it removed all cardinal arithmetic assumptions from the hypotheses.}
\begin{theorem}[Theorem 1 of Cox~\cite{CoxCoveringPaper}]\label{thm_CovThm}
Assume $0^\P$ does not exist, and let $K$ be the core model.  Suppose $\gamma$ is an ordinal, $\gamma > \omega_2$, $\text{cf}(\gamma) < |\gamma|$, and $\gamma$ is regular in $K$.  Then $\gamma$ is measurable in $K$.  Moreover, if $\text{cf}(\gamma) > \omega$ then in $K$, $\gamma$ has Mitchell order at least $\text{cf}^V(\gamma)$.
\end{theorem}

We now commence with the proof of Theorem \ref{thm_MainTheorem}.  Fix a large regular $\theta$ and a structure $\mathfrak{A} = (H_\theta, \in, \vec{S}, \dots)$ for the remainder of the proof.  For each $X$ witnessing mutual stationarity of $\vec{S}$, let $\sigma_X: H_X \to X \prec \mathfrak{A}$ be the inverse of the collapsing map of $X$ and let $K_X$ denote $\sigma_X^{-1}[K \cap H_\theta]$.

Recall that we are assuming $\mu^*= \liminf_{n \to \infty} \mu_n < \aleph_\omega$.  By Corollary \ref{cor_ReplaceBottom}, if we let $m_1$ be large enough so that $\omega_{m_1} > \mu^*$, then the mutual stationarity of $\langle S_n \ | \ n \ge m_1 \rangle$ is witnessed by sets containing $\mu^*$ as a subset; let $T$ denote this stationary set.  Lemma \ref{lem_Unif_Implies_Cover}, together with the fact that $\mu_n \ge \mu^*$ for all $n \ge m_1$ and $\mu^* \subset X$  for all $X \in T$, yields:

\begin{observation}\label{obs_OmegaClosedOrds}
For every $X \in T$, $X \cap \aleph_\omega$ is closed under limits of cofinality less than $\mu^*$.  In particular, since Theorem \ref{thm_MainTheorem} assumes that $\mu^* \ge \omega_1$, then $X \cap \aleph_\omega$ is an $\omega$-closed set of ordinals and thus Theorem \ref{thm_MitchellTheorem} applies.
\end{observation}


For $X \in T$ let $\beta^X_\omega:= \sigma_X^{-1}(\aleph_\omega)$.  By Observation \ref{obs_OmegaClosedOrds} and Theorem  \ref{thm_MitchellTheorem}, for every $X \in T$ the following facts hold for the coiteration of $K$ with $K_X ||\beta^X_\omega$:
\begin{equation}\label{eq_K_X_side_trivial}
\text{ the $K$ versus $K_X || \beta^X_\omega$ coiteration is trivial on the $K_X || \beta^X_\omega$ side} 
\end{equation}
and
\begin{equation}\label{eq_Truncates}
K \text{ truncates to a mouse of size at most } |\text{crit}(\sigma_X)| \text{ by stage 1}
\end{equation}



For each $X \in T$ and $n > m_1$ let
\begin{equation*}
\beta^X_n:= \sigma_X^{-1}(\omega_n) 
\end{equation*}
Since $\text{cf}(X \cap \omega_n) = \mu_n$ for all $n > m_1$, then  $\text{cf}^V(\beta^X_n) = \mu_n$.  So the assumptions of the theorem imply that for every $X \in T$:
\begin{equation}\label{eq_Vec_beta_not_Constant}
\langle \text{cf}^V(\beta^X_n) \ | \ n >m_1 \rangle \text{ is not eventually constant}
\end{equation}

Let $\theta_X$ denote the length of the coiteration of $K$ with $K_X || \beta^X_\omega$; equivalently, $\theta_X$ is the least stage of the $K$ versus $K_X$ coiteration such that all disagreements below $\beta^X_\omega$ have been resolved.  Let $\langle N^X_i, \kappa^X_i, \nu^X_i \ | \  i < \theta_X \rangle$ denote the mice, critical point, and iteration index of the mouse on the $K$-side of the coiteration of $K$ with $K_X || \beta_\omega^X$.  Note that by \eqref{eq_K_X_side_trivial} it follows that for all $i < \theta_X$:
\begin{equation}
\nu^X_i = \text{o}^{K_X}(\kappa^X_i)
\end{equation} 

The following argument is due to Magidor:
\begin{lemma}[Magidor~\cite{MR939805}]\label{lem_Magidor}
For every $X \in T$:
\begin{equation*}
\{ \kappa^X_i \ | \ i < \theta^X \} \cap \beta^X_\omega \text{ is cofinal in } \beta^X_\omega
\end{equation*}
\end{lemma}
\begin{proof}
  Assume not. By \eqref{eq_K_X_side_trivial} and universality of $K$,  $M^X_{\theta_X}$ end extends $K_X || \beta^X_\omega$.  Let $\eta_X$ be the strict supremum of $\{ \kappa_i^X \ | \ i < \theta_X \}$; by assumption, $\eta_X < \beta^X_\omega$.  Now \eqref{eq_Truncates} implies that $M^X_{\theta_X}$ projects below $\eta_X$ and is sound above $\eta_X$.\footnote{This is a routine inductive proof; see e.g.\ the proof of Lemma 6.6.4 of Zeman~\cite{ZemanBook}.}  Let $M$ be the maximal initial segment of $M^X_{\theta_X}$ such that $\beta^X_\omega$ is a cardinal in $M$.  If $M = M^X_{\theta_X}$ then we have already shown that there is some $\eta < \beta^X_\omega$ such that $M$ projects below $\eta$ and is sound above $\eta$.  If $M$ is a proper initial segment of $M^X_{\theta_X}$ then, since $\beta^X_\omega$ is definably collapsed over $M$, it follows that $M$ projects strictly below $\beta^X_\omega$ and, being a proper initial segment of a mouse, is (fully) sound.  In either case there are $n^*,m^* < \omega$ such that  
  \[ \rho^M_{n^*+1} \leq \beta^X_k < \beta^X_\omega \le \rho^M_{n^*} \]
  for all $k \geq m^*$.  Fix any $k > m^*$.  Since $\beta^X_k$ is regular in $K_X$, $\beta^X_\omega$ is a cardinal in $M$, and $M$ end-extends $K_X || \beta^X_\omega$, it follows by acceptability that $\beta^X_k$ is regular in $M$.  But then by Lemma \ref{lem_teclemma} together with the soundness properties of $M$ established above, $\cof(\beta^X_k) = \cof(\rho^M_{n^*})$ for all but finitely many $k$.  This contradicts \eqref{eq_Vec_beta_not_Constant}.
\end{proof}

Note that Lemma \ref{lem_Magidor} implies that:
\begin{equation}\label{eq_Theta_Limit_Ord}
\forall X \in T \ \ \theta_X \text{ is a limit ordinal}
\end{equation}

Lemma \ref{lem_Magidor}, together with the fact that there are only finitely many truncations in an iteration, yield that for every $X \in T$ there is an $n_X \in \omega$ such that, whenever $i < \theta_X$ and $\kappa^X_i \ge \beta^X_{n_X}$, then $i$ is not a truncation stage; i.e.\ all truncations of the $K$ versus $K_X ||\beta^X_\omega$ coiteration must occur before the critical points reach $\beta^X_{n_X}$.  Using \eqref{eq_Theta_Limit_Ord} and Fact \ref{fact_DegreesStabilize}, it follows that for each $X \in T$ the sequence
\begin{equation*}
\langle \degree(N^X_i, \kappa^X_i) \ | \ \kappa^X_i \ge \beta^X_{n_X} \ \text{and} \  i < \theta_X \rangle
\end{equation*}
is, eventually, a constant sequence of natural numbers.

So by increasing $n_X$ if necessary, we may also assume that $\degree(N^X_i, \kappa^X_i)$ is constant with value $m_X$ for all $i$ such that $\kappa^X_i \ge \beta^X_{n_X}$.  By countable completeness of the nonstationary ideal:
\begin{equation}\label{eq_m_star_n_star}
\exists m^*, n^* \in \omega \ \ \exists T' \subset T \text{  stationary } \forall X \in T' \ n_X = n^* \ \text{and} \ m_X = m^*
\end{equation}

Let $X \in T'$.  Since (total) iteration maps are cofinal, we have that the cofinality of $\rho_{m^*}(N^X_i)$ is constant for all $i$ which satisfy:
\begin{equation}\label{eq_NoDropInterval}
\beta^X_{n^*} \le \kappa^X_i < \beta^X_\omega
\end{equation}
For each $X \in T'$ let $\lambda_X$ denote the constant cofinality of $\rho_{m^*}(N^X_i)$, for those $i$ satisfying (\ref{eq_NoDropInterval}).   

For each $n \in \omega$ define:
\begin{equation}\label{eq_Theta_X_n}
\theta^X_n:= \text{ the least stage such that } \kappa^X_{\theta^X_n} \ge \beta^X_n
\end{equation}
Note that:
\begin{equation}
\forall n \in \omega \ \ \Big|K_X || \beta^X_n\Big| = \Big| X \cap \omega_n \Big| < \aleph_\omega
\end{equation}
Combined with \eqref{eq_Truncates} and Lemma 4.4.1 of Zeman~\cite{ZemanBook}, 
 this implies that $|N^X_i| < \aleph_\omega$ for all $i \in (1, \theta_X)$.  In particular:
\begin{equation}
\forall X \in T' \ \ \lambda_X < \aleph_\omega
\end{equation}

So $\lambda_X \in \{ \omega_k \ |\ k \in \omega  \} \subset X$.  Thus by pressing down there is some fixed infinite cardinal $\lambda^* < \aleph_\omega$ and a stationary $T'' \subset T'$ such that $\lambda_X = \lambda^*$ for all $X \in T''$.  Since $\langle \mu_n \ | \ n > m_1 \rangle$ is not eventually constant:
\begin{equation}\label{eq_InfinitelyOften}
I:= \{ n \in \omega \ | \  \lambda^* \ne \mu_n \big( = \text{cf}^V(\beta_n) \big)  \}  \text{ is infinite}
\end{equation}

We now consider two cases.  If, for some $X \in T''$ and $n \in I \cap (n^*, \omega)$, there is an iterate $N^X_i$ such that $\text{crit}(E^X_i) < \beta^X_n$ but the generators of $E^X_i$ are cofinal in $\beta^X_n$, then by iterating this extender we can obtain a model as in the conclusion of Theorem \ref{thm_MainTheorem}.  \textbf{So from now on we assume there is no such extender}, i.e.\ assume: 
\begin{eqnarray}\label{eq_OnlyOneGenerator}
\begin{split}
\forall X \in T'' \ \  \forall n \in I \cap (n^*, \omega) \ \  \forall i < \theta^X_n  \\
\text{ the generators of } E^X_i \text{ are bounded below } \beta^X_n
\end{split}
\end{eqnarray}

\begin{lemma}\label{lem_CofInBeta_n}
For every $X \in T''$ and for all $n \in I \cap (n^*, \omega)$: the critical points of the coiteration are cofinal in $\beta_n^X$.  
\end{lemma}
\begin{proof}
Fix $n > n^*$ s.t. $n \in I$; i.e.\ $\cof(\beta^X_n) \neq \lambda^* = \lambda_X$. Now let us assume for a contradiction that there is an $i$ with $\kappa^X_i < \beta^X_n$ but $\kappa^X_{i+1} \geq \beta^X_n$. Note that $\beta^X_n$ is regular in $K_X$.  Since $K_X || \beta_\omega^X$ doesn't move in the coiteration, $i > n^*$ (in particular $i$ isn't a truncation stage), and by acceptability, it follows that $\beta^X_n$ is also regular in $N^X_{i+1}$. Furthermore, assumption \eqref{eq_OnlyOneGenerator} implies that the generators of $E^X_i$ are bounded by some $\zeta < \beta^X_n$, which in turn implies that $N^X_{i+1}$ is sound above $\zeta + 1$.  Also $\rho_{m^*}(N^X_{i+1}) \geq \kappa^X_{i+1} \geq \beta^X_n$ (recall $m^*$ was defined in \eqref{eq_m_star_n_star} as the uniform eventual value of $\degree(N^X_j, \kappa^X_j)$). So we can apply Lemma \ref{lem_teclemma} to conclude that $\cof(\beta^X_n) = \lambda_X$. But this contradicts our choice of $n$!
\end{proof}

In particular if $X \in T''$ and $n \in I \cap (n^*,\omega)$, then $\theta_n^X$ is a limit ordinal and $\text{cf}^V(\theta^X_n) = \text{cf}^V(\beta^X_n) = \mu_n$; here $\theta^X_n$ is as defined in \eqref{eq_Theta_X_n}.

\begin{lemma}\label{lem_thread}
Let $X \in T''$ and $n  \in I \cap (n^*, \omega)$. Then the following set is closed and unbounded in $\theta^X_n$:
\begin{equation*}
C^X_n:= \{ j < \theta^X_n \ | \  \pi^X_{j, \theta^X_n}(\kappa^X_j) = \beta^X_n \} 
\end{equation*} 
\end{lemma}

\begin{proof}
First we show that $C^X_n$ is unbounded in $\theta^X_n$.  Assume not, and let $i_0 < \theta^X_n$ be a bound on $C^X_n$.  By Lemma \ref{lem_CofInBeta_n}, $\theta^X_n$ is a limit ordinal.  So there is some $j^* \in (i_0,  \theta^X_n)$ such that $\beta^X_n$ has a preimage in $N^X_{j^*}$, say $\bar{\beta}$.  We claim that
\begin{equation}\label{eq_LessBarBeta}
\kappa^X_{j^*} < \bar{\beta}
\end{equation} 
Suppose not.  Our assumptions imply that these two ordinals are not equal, so it must be that $\kappa^X_{j^*} > \bar{\beta}$.  But $\kappa^X_{j^*} < \beta^X_n$ (since $j^* < \theta^X_n$), so since $\pi^X_{j^*,\theta^X_n} \restriction \kappa^X_{j^*} = \text{id}$ this would imply that $\beta^X_n < \beta^X_n$, a contradiction.  
 
Since $\beta^X_n$ is regular in $K_X || \beta^X_\omega$, $\theta^X_n$ is past all truncation points of the $K$ versus $K_X || \beta^X_\omega$ coiteration, and $K_X$ does not move in the coiteration, it follows that $\beta^X_n$ is regular in $N^X_{\theta^X_n}$.  So by elementarity of the iteration map:
\begin{equation}\label{eq_BarBetaRegular}
\bar{\beta} \text{ is regular in } N^X_{j^*}
\end{equation}
So, our iteration embeddings are continuous at $\bar{\beta}$ and thus $\cof(\bar{\beta}) = \cof(\beta^X_n) \neq \lambda_X$, where the latter inequality is because $n \in I$. 
 
On the other hand $N^X_{j^*}$ is sound above $\kappa^X_{j^*}$, $\bar{\beta}$ is regular in $N^X_{j^*}$ by \eqref{eq_BarBetaRegular}, and $\bar{\beta}$ is strictly above $\kappa^X_{j^*}$ by \eqref{eq_LessBarBeta}.  So we can conclude by Lemma \ref{lem_teclemma} that $\cof(\bar{\beta}) = \lambda_X$.  This is a contradiction, completing the proof that $C^X_n$ is unbounded.  That $C^X_n$ is closed below $\theta^X_n$ is a routine exercise, using the fact that the critical points of the iteration are increasing.
\end{proof}

Let $I'$ denote the tail end of $I$ beyond $n^*$, and also ensure that 
\[
(\mu^*)^+ < \omega_{\text{min}(I')}
\]

For the rest of the proof, fix some $n \in I'$; by \eqref{eq_InfinitelyOften} there are infinitely many such $n$.  Also fix some $X \in T''$.  Observe that if $C^X_n$ is as in the statement of Lemma \ref{lem_thread}, then 
\begin{equation*}
D^X_n:= \{ \alpha \ | \ \exists j \in C^X_n \ \ \alpha = \kappa^X_j \} \text{ is club in } \beta_n 
\end{equation*}

Also observe that if $j \in C^X_n$ then since $j$ is past all truncations, $\kappa^X_j$ is a regular cardinal in the $j$-th iterate of $K_X$; but since $K_X$ doesn't move in the coiteration this just means $\kappa^X_j$ is regular in $K_X$.  Thus
\begin{equation*}
\forall \alpha \in D^X_n \ \ K_X \models \ \alpha \text{ is regular}
\end{equation*} 
and so by elementarity of $\sigma_X$ it follows that:
\begin{equation}\label{eq_WideTildeReg}
\forall \alpha \in \widetilde{D}^X_n:= \sigma_X[D^X_n] \ \ \ K \models \alpha \text{ is regular}
\end{equation}

By Observation \ref{obs_OmegaClosedOrds}, $\widetilde{D}^X_n$ is closed under limits of cofinality $<\mu^*$.  Also $\widetilde{D}^X_n$ is cofinal in $\text{sup}\big( \sigma_X [ \beta_n] \big) = \text{sup}(X \cap \omega_n)$.  Together with \eqref{eq_WideTildeReg} it follows that
\begin{align}\label{eq_AllLimitsOfReg}
\begin{split}
\forall \eta \in \text{lim} \big( \widetilde{D}^X_n \big) \cap \text{cof}(\ge \mu^*)  \text{, all but nonstationarily} \\
\text{many members of } \eta \cap \text{cof}(<\mu^*) \text{ are regular in } K
\end{split}
\end{align} 
The notation $\text{lim} \big( \widetilde{D}^X_n \big) \cap \text{cof}(\ge \mu^*)$ in \eqref{eq_AllLimitsOfReg} really means \textbf{all} limits of $\widetilde{D}^X_n$ of cofinality $\ge \mu^*$, not just those below $\text{sup}(X \cap \omega_n)$.  In particular, it includes the ordinal $\text{sup}(X \cap \omega_n)$.\footnote{And $\text{sup}(X \cap \omega_n)$ might be the \emph{only} element of $\text{lim} \big( \widetilde{D}^X_n \big) \cap \text{cof}(\ge \mu^*)$, in the case that $\mu^* = \mu_n$.} 

\begin{globalClaim}\label{clm_EtaLargeOrder}
If $\eta \in \text{lim} \big( \widetilde{D}^X_n \big) \cap \text{cof}(\ge \mu^*)$ \textbf{and} $\text{cf}(\eta) < \omega_{n-1} \le \eta$, then $o^K(\eta) \ge \text{cf}^V(\eta) \ge \mu^*$.
\end{globalClaim}
\begin{proof}
Fix such an $\eta$.  The assumptions of the claim guarantee that $\omega < \text{cf}(\eta) < |\eta|$ and that $\eta > \omega_2$; so by Theorem \ref{thm_CovThm}, to prove that $o^K(\eta) \ge \text{cf}^V(\eta)$ it suffices to prove that $\eta$ is regular in $K$.  Suppose for a contradiction that $\eta$ is singular in $K$.  In $K$, fix some continuous $\vec{\eta} = \langle \eta_i \ | \ i < \text{cf}^K(\eta) \rangle$ which is cofinal in $\eta$ and such that $\eta_0 > \text{cf}^K(\eta)$.  Then every member of
\[
E:= \{ \eta_i \ | \  i \text{ is a limit ordinal}  \}
\]
is singular in $K$,\footnote{Because $\vec{\eta} \restriction i$ witnesses singularity of $\eta_i$.} and moreover $E$ is club in $\eta$.  So in particular, almost every member of $\eta \cap \text{cof}(<\mu^*)$ is singular in $K$.  This contradicts \eqref{eq_AllLimitsOfReg}.
\end{proof}

\begin{globalClaim}\label{clm_StatMany}
The set of $\eta$ which satisfy the assumptions of Claim \ref{clm_EtaLargeOrder} is stationary in $\omega_n$.
\end{globalClaim}
\begin{proof}
Note that $\mu_n \ge \mu^*$; we consider two cases, depending on whether this inequality is strict.  

If $\mu_n = \mu^*$ then $\text{sup}(X \cap \omega_n)$ is a $\mu^*$-cofinal limit of $\widetilde{D}^X_n$.\footnote{Possibly the only such limit of $\widetilde{D}^X_n$; i.e.\ in the case $\mu_n = \mu^*$, then at most nonstationarily many members of $\widetilde{D}^X_n$ are $\mu^*$-cofinal.}  Also, since $n \in I'$ then $\mu^* < \omega_{n-1}$, and so the cofinality of $\text{sup}(X \cap \omega_n)$ is strictly less than $\omega_{n-1}$.  Finally, note that 
\begin{equation*}
\bigcup_{X \in T''} \{ \text{sup}(X \cap \omega_n) \} 
\end{equation*}
is stationary in $\omega_n$, because $T''$ is stationary.

If $\mu_n > \mu^*$, then $Q^X_n:= \text{lim}(\widetilde{D}^X_n) \cap \text{cof}(\mu^*) \cap [\omega_{n-1},  \text{sup}(X \cap \omega_n))$ is stationary (in fact $\mu^*$-club) in $\text{sup}(X \cap \omega_n)$ for all $X \in T''$.  Also since $n \in I'$ then $\mu^* < \omega_{n-1}$.  It follows that every $\eta \in Q^X_n$ satisfies the assumptions of Claim \ref{clm_EtaLargeOrder}.  Finally, note that because each $Q^X_n$ is stationary in $\text{sup}(X \cap \omega_n)$ and $T''$ is stationary, it follows that  
\begin{equation*}
\bigcup_{X \in T''} \ \  Q^X_n 
\end{equation*}
is stationary in $\omega_n \cap \text{cof}(\mu^*)$, which completes the proof of the claim.
\end{proof}

Thus Claims \ref{clm_EtaLargeOrder} and \ref{clm_StatMany} imply that for any $n \in I'$, there are stationarily many $\eta < \omega_n$ such that $o^K(\eta) \ge \mu^*$.  This completes the proof of  Theorem \ref{thm_MainTheorem}.

%% file: sec_FurtherTheorems.tex
\section{Stronger Hypotheses}\label{sec_StrongerHypotheses}

In this section we shall prove Theorems \ref{thm_MainTheorem_two},\ref{thm_MainTheorem_three} and \ref{thm_MainTheorem_four}. Let us start with Theorem \ref{thm_MainTheorem_two}. Let $l$,$m$,$\<S_n : n > m + 1\>$, $\<n_k: k < \omega\>$ be as in its statement. As before we can find a stationary set $T$ of $X \subseteq H_{\aleph_\omega}$ s.t $\sup(X \cap \aleph_n) \in S_n$ for all $n > m + 1$ and in the coiteration of $K_X$ with $K$, which we can assume to be linear in this case, the $K$-side of the iteration drops and the $K_X$-side is trivial.

As before iteration indices are cofinal in $\beta^X_\omega$ and hence $\theta_X$ is a limit for all $X$. So we can fix an $n^*$ such that whenever $\nu^X_i \geq \beta^X_{n^*}$ then there is no drop between $i$ and $\theta_X$. Also remember that whenever $k > n^*$ and $j < l$, then $\mu_{n_k} = \mu_{n_k + j}$. For such $i$ that $\mu^X_i \geq \beta^X_{n^*}$ let us call the degree of elementarity of $\pi^X_{i,i+1}$ at that point $m^*$ and let us refer to the - constant in $i$ - cofinality of $\rho_{m^*}(N^X_i)$ as $\lambda_X$. Then there exist infinitely many $k > n^*$ s.t. $\mu_{n_k} \neq \lambda_X$.

An important difference is that we can no longer prove iteration indices to be cofinal in $\beta_{n_k}$ even if $\mu_{n_k} \neq \lambda_X$. In fact, we will show that this is not the case! This is because our extenders might now have many generators. 

\begin{observation}\label{obs_succcof}
Let $k \geq m + 1$. Let $\alpha \in \left[\beta^X_{n_k},\beta^X_{n_k + l}\right)$ be s.t. $K^X \models \exists \gamma: \alpha = \gamma^+$. Then $\cof(\alpha) = \mu_{n_k}$.
\end{observation}

\begin{proof}
If $\alpha = \beta^X_{n_k + j}$ for some $j < l$ then this is by choice of our sequence. If not, then $\alpha$ is properly in between say $\beta^X_{n_k + j}$ and $\beta^X_{n_k + j + 1}$ and thus by weak covering $\cof(\sigma_X(\alpha)) = \aleph_{n_k + j}$. W.l.o.g $X$ is closed under some function witnessing this. But this easily gives $\cof(\alpha) = \cof(\beta^X_{n_k + j}) = \mu_{n_k + j} = \mu_{n_k}$.
\end{proof}

\begin{lemma}
 Let $k > n^*$ be s.t. $\mu_{n_k} \neq \lambda_X$. Then there exist an $i < \theta_X$ s.t. $\kappa^X_i  < \beta^X_{n_k} \leq \nu^X_i$.
\end{lemma}

\begin{proof}
Assume not. Because iteration indices are cofinal in $\beta^X_\omega$ there is some least $i$ s.t $\nu^X_i > \beta^X_{n_k}$. By assumption we have $\kappa^X_i \geq \beta^X_{n_k}$. Then by coherence and the fact that there is no drop in between $i$ and $\theta_X$ means that $((\beta^X_{n_k})^+)^{K_X} = ((\beta^X_{n_k})^+)^{M^X_i}$ is a regular cardinal of $N^X_i$. Furthermore, because by minimality of $i$ all generators of the iteration up to this point are less than $\beta^X_{n_k}$, $N^X_i$ is sound above $\beta^X_{n_k}$. Lastly, $((\beta^X_{n_k})^+)^{K_X} \leq \rho_{m}(N^X_i)$ because there is no drop at $i$. So, Lemma \ref{lem_teclemma} applies and gives us that $\cof(((\beta^X_{n_k})^+)^{K_X}) = \lambda_X$. On the other hand by Observation \ref{obs_succcof} $\cof(((\beta^X_{n_k})^+)^{K_X}) = \mu_{n_k}$. Contradiction!
\end{proof}

Note here that by the same proof we have that $\cof(((\kappa^X_i)^+)^{N^X_i}) = \lambda_X$ and thus it should be easy to see that $((\kappa^X_i)^+)^{N^X_i} < \beta^X_{n_k}$.

So for any $k > n^*$ s.t. $\mu_{n_k} \neq \lambda_X$ we can fix some $i^X_k$ with $\kappa^X_{i^X_k} < \beta^X_{n_k}$ and $\nu^X_{i^X_k} \geq \beta^X_{n_k}$. To simplify our notation we shall henceforth refer to $\kappa^X_{i^X_k}$ as $\eta^X_k$, to $\nu^X_{i^X_k}$ as $\zeta^X_k$, to the model $N^X_{i^X_k}$ as $M^X_k$ and to the extender $E^X_{i^X_k}$ as $F^X_k$.

\begin{lemma}\label{lemma_nuplus}
Let $k > n^*$ be s.t. $\mu_{n_k} \neq \lambda_X$. Then $((\zeta^X_k)^+)^{K_X} \geq \beta^X_{n_k + l}$.
\end{lemma}

\begin{proof}
Consider $M^* := \ult(M^X_k,F^X_k)$. In $M^*$, $((\zeta^X_k)^+)^{N^*}$, which equals $((\zeta^X_k)^+)^{K_X}$ by coherence, is certainly regular and $M^*$ is sound above $\zeta^X_k$. Notice also that $\rho_{m^*}(M^*) \geq ((\zeta^X_k)^+)^{M^*}$. So, by Lemma \ref{lem_teclemma} $\cof(((\zeta^X_k)^+)^{M^*}) = \lambda_X$. On the other hand $((\zeta^X_k)^+)^{K_X}$ is a successor. card. in $K_X$; if it were in the interval $\left[\beta^X_{n_k},\beta^X_{n_k + l}\right)$, by Observation \ref{obs_succcof} it's cofinality would equal $\mu_{n_k}$. So, we can conclude that $((\zeta^X_k)^+)^{K_X} \geq \beta^X_{n_k + l}$.
\end{proof}

%

We can immediately conclude that $\mitord^{K_X}(\eta^X_k) \geq \beta^X_{n_k + (l - 1)}$ for every  $k > n^*$ s.t. $\mu_{n_k} \neq \lambda_X$. By elementarity then $\mitord^K(\sigma_X(\eta^X_k)) \geq \aleph_{n_k + (l - 1)}$. Also, we know there exists infinitely many such $k$. So this concludes the proof of Theorem \ref{thm_MainTheorem_two}.

The proof gives a slightly stronger conclusion! 

\begin{fact}\label{thefanciestoffactums}
Assume that $\xi \in \left(\sigma_X(\eta^X_k),\aleph_{n_k + (l-1)}\right) \cap \card^K$, then $\mitord^K(\sigma_X(\eta^X_k)) \geq \xi$.
\end{fact}

We will need the above fact for the proof of Theorem \ref{thm_MainTheorem_three}:

\begin{proof}[Proof of Theorem \ref{thm_MainTheorem_three}]
We will do the proof for cofinalities $\aleph_1$ and $\aleph_2$, it is not hard to see that this case is representative. We just need to consider two sequences $\vec{S} := \<S^n_{f(n)} : n \geq 8\>$ and $\vec{T}:=\<S^n_{g(n)} : n \geq 4\>$ where 
\begin{align*}
  f(n) & = \begin{cases} 1 & n \mod 8 = 0,1,2,3 \\ 2 & n \mod 8 = 4,5,6,7 \end{cases} \\
  g(n) & = \begin{cases} 1 & n \mod 4 = 0,1 \\ 2 & n \mod 4 = 2,3 \end{cases}
\end{align*}
Assume both $\vec{S}$ and $\vec{T}$ are mutually stationary. Using Theorem \ref{thm_MainTheorem_two} we get a sequence $\<\kappa_n: n <\omega\>$ and $\<\lambda_n: n< \omega\>$ s.t. for all $n < \omega$ there exists $\kappa < \kappa_n$ and $\kappa' < \lambda_n$ with $\mitord^K(\kappa) \geq \kappa^{+3}_n$ and $\mitord^K(\kappa') \geq \lambda^+_n$.

As to the identity of the $\kappa_n$'s they are the $\aleph_k$'s with either $k$ at least some number $n^*$ and $k \mod 8 = 0$ or $k \geq n^*$ and $k \mod 8 = 4$. Similarly, the $\lambda_n$'s are the $\aleph_k$'s with either $k$ at least some number $n^*$ and $k \mod 4 = 0$ or $k \geq n^*$ and $k \mod 4 = 2$. Of course, we can assume the two $n^*$'s to be the same.

Our job is now to simply check all of the 4 possible combinations and see that there must be some overlap on the $K$-sequence. By symmetry it suffices to examine just two of those cases.

 Take some $k$ big enough with $k \mod 8 = 0$. Assume there is some $\kappa < \aleph_k$ with $\mitord^K(\kappa) \geq \aleph_{k+3}$. The first case we look at is that there is $\kappa' < \aleph_{k + 2}$ with $\mitord^K(\kappa') \geq \aleph_{k + 3}$. This then tells us that there must be some other $\kappa'' < \aleph_{k - 2}$ with $\mitord^K(\kappa'') \geq \aleph_{k - 1}$. In our situation we have that $\kappa < \aleph_k$ is a regular cardinal in $K$ thus by Fact \ref{thefanciestoffactums} we have that $\mitord^K(\kappa'') \geq \kappa$. If $\nu$ was the index of the order zero measure on $\kappa$ then $K \vert\vert \nu$ is a $0^\P$ type mouse. 

The other case works similar. Assume now that $\kappa' < \aleph_k$ with $\mitord^K(\kappa') \geq \aleph_{k+1}$ exists. Then we also have $\kappa'' < \aleph_{k+4}$ measurable in $K$. As before - but applying Fact \ref{thefanciestoffactums} at $\kappa$ instead - we actually have $\mitord^K(\kappa) \geq \kappa''$ and thus $0^\P$.

As mentioned before, the remaining two cases are dealt with by a symmetric argument.
\end{proof}

Finally, the proof of theorem \ref{thm_MainTheorem_four}:

So let us fix $ 1 < m < \omega$ and a mutually stationary sequence $\vec{S}:= \<S_n: n > m\>$ s.t for all $n$, $S_n$ concentrates on a fixed cofinality $\mu_n$ s.t. $\liminf\limits_{m < n < \omega} \mu_n = \aleph_\omega$. It is easy to see that we can require all the $\mu_n$ to be uncountable.

We shall do the following proof in greater generality. The above hypothesis is almost certainly very strong, close to inconsistent even. We believe it should be possible to extract an inner model with a Woodin cardinal from the hypothesis. Considering that the consistency of the statement is unsure, it might not be a worthy pursuit to do so. 

We assume for a contradiction:

\begin{itemize}
 \item[(a)] $K$ is a core model satisfying weak covering at all but finitely many cardinals;
 \item[(b)] if $E$ is a total extender on the $K'$-sequence where $K' \eextend K$, $\kappa$ is it's critical point and $\nu$ it's index, then $\nu$ is a successor cardinal in $\ult(K';E)$ and $\cof((\kappa^+)^{K'}) = \cof(\nu)$;
 \item[(c)] there exists some $X \prec (H_{\omega};\in, K \cap H_\omega,\ldots)$ s.t. $\sup(X \cap \aleph_n) \in S_n$ for all $m < n <\omega$ and in the co-iteration of $K$ and $K_X$ which is not necessarily linear, $K_X$ does not move and $K$ drops along its main branch.
\end{itemize}

The above is satisfied if $0^\P$ does not exist as evidenced by the core model below $0^\P$(, except (b) which is not quite true, but we can make do by substituting $(\nu^+)^{\ult(K',E)}$ for $\nu$, the former does have the right cofinality as shown in the proof of Lemma \ref{lemma_nuplus}). We do not know if it is satisfied if there is no inner model with a Woodin cardinal.

So, let us write $\mathcal{T}_X$ for the iteration tree on $K$ and $b_X$ for its main branch from assumption (c). Let $\theta_X$ be the length of $\mathcal{T}_X$, $\<N^X_i,\kappa^X_i,\nu^X_i,m^X_i: i < \theta_X\>$ be the iteration's models, critical points, indices and degrees.  As before we can show that $b_X$ has limit type. So there is some $n^*$ s.t whenever $\nu^X_i \geq \beta^X_{n^*}$ and $i \in b$ then $\cof(\rho_{m^X_i}(N^X_i))$ is constant in $i$. Call this constant value $\lambda_X$. W.l.o.g. $\lambda_X < \mu_n$ for all $n > n^*$.

\begin{observation}\label{obs_succcof_two}
Let $n > n^*$. Let $\alpha \in [\beta^X_n,\beta^X_{n+1})$ be s.t. $K^X \models \exists \gamma :\alpha = \gamma^+$. Then $\cof(\alpha) > \lambda_X$.
\end{observation}

\begin{proof}
If $\alpha = \beta^X_{n+1}$ then this is by choice of our sequence. If not, then by weak covering $\cof(\sigma_X(\alpha)) = \aleph_n$. W.l.o.g $X$ is closed under some function witnessing this. But this easily gives $\cof(\alpha) = \cof(\beta^X_n) = \mu_n > \lambda_X$.
\end{proof}

We can now derive a contradiction finishing the proof of theorem \ref{thm_MainTheorem_four}:

Let $i + 1 \in b$ be s.t. $\nu^X_i \geq \beta^X_{n^*}$. On the one hand we have that $\nu^X_i$ is a successor cardinal of $K_X$. Thus by Observation \ref{obs_succcof_two} $\cof(\nu^X_i) > \lambda_X$. 

On the other hand by assumption (b) $\cof(\nu^X_i) = \cof(((\kappa^X_i)^+)^{M^X_i})$; furthermore, $((\kappa^X_i)^+)^{M^X_i} = ((\kappa^X_i)^+)^{M^X_j}$ because of agreement between models in iteration trees and, crucially, the fact that there occur no more drops on $b_X$ from this stage on. (Here $j$ is the $\mathcal{T}^X$-predecessor of $i + 1$). Clearly though, $((\kappa^X_i)^+)^{M^X_j}$ is a regular cardinal of $M^X_j$, and that model is sound above $\kappa^X_i$. So, Lemma \ref{lem_teclemma} applies and gives $\cof(((\kappa^X_i)^+)^{M^X_i}) = \lambda_X$. Hence $\cof(\nu^X_i) = \lambda_X$. Contradiction!

%% file: sec_OpenProblems.tex
\section{Open Problems}\label{sec_OpenProblems}

\begin{question}
 Is it possible to force, starting from a model with at most finitely many measurable cardinals, that the sequence $\<S^2_0,S^3_1,S^4_1,S^5_0,S^6_1,S^7_1,\ldots\>$ is mutually stationary?
\end{question}

\begin{question}
 What is the upper bound for the existence of a mutually stationary sequence satisfying the hypothesis of Theorem \ref{thm_MainTheorem_two}?
\end{question}

\begin{question}
 Is the hypothesis of Theorem \ref{thm_MainTheorem_three} consistent relative to large cardinals?
\end{question}

\begin{question}
 Does ``$\aleph_\omega$ is Jonsson" imply that there exists in a - possibly trivial - forcing extension $V\left[G\right]$ a mutually stationary sequence satisfying the hypothesis of Theorem \ref{thm_MainTheorem_four} relative to $V\left[G\right]$?
\end{question}

\begin{question}\label{question_5}
 Is it possible to generate mutually stationary sequences not coverded by Theorem \ref{thm_LiuShelah_1}, e.g. the sequence $\<S^4_1,S^5_1,S^6_2,S^7_2,S^8_1,S^9_1,S^10_2,S^11_2,\ldots\>$, using Theorem \ref{PCF}, i.e. is it possible to have $\cof(\prod\limits_{n \in A_0} \aleph_n) = \aleph_{\omega + 1}$ and $\cof(\prod\limits_{n \in A_1} \aleph_n) = \aleph_{\omega + 2}$ where $A_0 := \{ n < \omega\vert n \mod 4 = 0,1\}$ and $A_1 := \{ n < \omega \vert n \mod 4 = 2,3\}$ or vice versa?
\end{question}